\documentclass[12pt]{article}
\usepackage{amsmath,amsfonts,amssymb,amsthm}

\theoremstyle{definition}
\newtheorem{defi}{Definition}[section]
\newtheorem{ex}[defi]{Example}
\newtheorem{rem}[defi]{Remark}

\theoremstyle{plain}
\newtheorem{thm}[defi]{Theorem}

\newtheorem{cor}[defi]{Corollary}
\newtheorem{lmm}[defi]{Lemma}

\numberwithin{equation}{section}

\newcommand{\ve}{\varepsilon}

\newcommand{\R}{\mathbb R}

\newcommand{\E}{\textsf E}

\newcommand{\Sub}{\mathrm{Sub}}

\renewcommand{\P}{\mathbf{P}}
\renewcommand{\d}{\textrm{d}}

\newcommand{\Expect}{\mathsf{E}}

\begin{document}
\begin{center}
{\Large\bf Estimates for functionals of solutions to higher-order heat-type equations with random initial conditions}
 \end{center}

\begin{center}
{\large\bf Yu. Kozachenko\footnote{Department of Probability Theory, Statistics and Actuarial Mathematics, Taras Shevchenko National University of Kyiv, 64 Volodymyrska str., 01601, Kyiv,  Ukraine; ykoz@ukr.net}, E. Orsingher\footnote{Department of Statistical Sciences, Sapienza University of Rome, P.le Aldo Moro, 5, 00185, Rome, Italy; enzo.orsingher@uniroma1.it},
 L. Sakhno\footnote{Department of Probability Theory, Statistics and Actuarial Mathematics, Taras Shevchenko National University of Kyiv, 64 Volodymyrska str., 01601, Kyiv,  Ukraine; lms@univ.kiev.ua},
  O. Vasylyk\footnote{Department of Probability Theory, Statistics and Actuarial Mathematics, Taras Shevchenko National University of Kyiv, 64 Volodymyrska str., 01601, Kyiv,  Ukraine; ovasylyk@univ.kiev.ua, +38-097-877-28-70}}
\end{center}

\medskip

\begin{abstract}
\noindent In the present paper we continue the investigation of solutions to higher-order heat-type equations with random initial conditions, which play the important role in many applied areas. We consider the random initial conditions given by harmonizable $\varphi$-sub-Gaussian processes.  The main results are  the bounds for the distributions of the suprema over bounded and unbounded domains for solutions of such equations. The results obtained in the paper hold, in particular,  for the case of Gaussian initial condition.
\\ \\
\textit{Key words}: higher-order heat-type equations, random initial conditions, harmonizable processes, $\varphi$-sub-Gaussian processes, distribution of sumpremum
\end{abstract}

\section{Introduction}

Partial differential equations of the form
\begin{equation}\frac{\partial u}{\partial t} = L\,u,\label{1.1}
\end{equation}
where $L$ is a linear or nonlinear operator containing higher order spatial derivatives, often appear in the literature, since they can model many interesting phenomena in physics and other applied areas. Equations of this type represent higher order extensions of the heat equation and also are called  evolution equations. They are used, in particular, to describe the wave propagation in fluids, plasma and other media. Higher order nonlinear and/or linear terms can play significant role in adequate and accurate representation of complex physical systems.

The literature devoted to the study of this type of partial differential equations is rather voluminous, some references relevant to our study will be given below.  In recent decades, the subject has expanded even more intensively due to the discovery of interesting and fruitful probabilistic connections.

One famous equation is the Burgers equation
\begin{equation}\label{1.2}
\frac{\partial u}{\partial t}=\frac{1}{2}\frac{\partial^2 u}{\partial x^2}+u\frac{\partial u}{\partial x}
\end{equation}
which is a second order counterpart of the Korteweg-De Vries (KdV)  equation
\begin{equation}\label{1.3}
\frac{\partial u}{\partial t}=-\frac{\partial^3 u}{\partial x^3}+u\frac{\partial u}{\partial x}.
\end{equation}

The Burgers equation is mathematically tractable because the Cole-Hopf transformation reduces it to the heat equation. The KdV equation admits  a number of exact analytic solutions, in particular, the so-called soliton solution, and is important in fluid dynamics.
We also mention  that numerous recent studies are devoted to the higher-order generalizations of the KdV equation.

Other non-linear equations of the form \eqref{1.1} appear in the description of diffusions with branching and are called the Kolmogorov-Petrovskii-Piskunov equations (sometimes called also Fisher equations) and have the form
\begin{equation}\label{1.4}
\frac{\partial u}{\partial t}=\frac{1}{2}\frac{\partial^2 u}{\partial x^2}+u(1-u).
\end{equation}

The list of non-linear relevant partial differential equations is huge and includes also the fractional equations
\begin{equation}\label{1.5}
\frac{\partial u}{\partial t}=-(-\Delta)^\nu \, u.
\end{equation}
with $-(-\Delta)^\nu$ being the fractional Laplacian, related to stable processes.

The investigation of partial differential equations  (PDE) from the probabilistic point of view has been performed in several directions. From one hand,  it was shown that many  models of stochastic processes possess probability laws linked to various types of PDE's. On the other hand, many interesting results have been obtained for PDE subject to random initial conditions, random boundary conditions or external random forces.

 \vspace*{0.2cm}
\noindent{\it 1.1.  Pseudoprocesses related to higher-order heat-type equations} \vspace*{0.2cm}

Higher-order linear equations of the form
\begin{equation}\label{1.6}
\frac{\partial u}{\partial t}=c_{n}\frac{\partial ^{n}u}{\partial x^{n}%
}\;,\qquad n\ge 2
\end{equation}
have been analysed by many authors as possible extensions of the classical heat equation.

The fact that fundamental solutions to the equation \eqref{1.6} are sign-varying functions dates back to works by Bern\v{s}tein and L\'{e}vy.

Processes constructed as the Wiener processes by means of the fundamental solutions of \eqref{1.6} (called pseudoprocesses because of the sign-varying character of their measure densities) have been actively investigated starting from the works by Krylov \cite{Krylov}, Daletsky, Fomin \cite{Dal65}, \cite{Dal69}, Ladohin \cite{Ladohin}, Miyamoto \cite{M}, Hochberg \cite{H}.

More recently pseudoprocesses related to higher-order equations have been analysed and functionals connected with them have been evaluated by Orsingher \cite{Orsingher1991}, Hochberg and Orsingher \cite{HO}, \cite{HochbergOrsingher1996}, see also \cite{BHO}, and, in more systematic way, by Lachal \cite{Lachal2003}-\cite{Lachal2014} among others.
Fractional versions of the equation \eqref{1.6} with space derivatives of the Riemann-Liouville type have been studied by Orsingher and Toaldo \cite{Orsingher2014}, Smorodina, Faddeev, Platonova \cite{ISF}, \cite{Smorodina}, \cite{Platonova}.

Many of the papers mentioned above study the distributions of functionals of the pseudoprocess $X_t$, such as the sojourn time $\Gamma_t=\int_0^t I_{[0,\infty)} (X_s)ds$, the local times and others.

In some cases functionals like $\underset{0\le s\le t}{\max}  X_s$, $X(T_{a})$, $X(T_{a,b})$, where $T_{a}=\inf\{s: X_s\ge a\}$,  $T_{a,b}=\inf\{s: (X_s\ge b)\cup (X_s\le a)\}$,  permit us to go deeper into the sample behavior of the pseudoprocesses (see, for example,  some papers by Lachal \cite{Lachal2003}-\cite{Lachal2014}, Nishioka \cite{Nishioka1996}-\cite{Nishioka2001}).

Pseudoprocesses constructed as limit of pseudo random walks have been proposed recently by Nakajima and Sato \cite{NS} and before by Lachal \cite{Lachal2014}.

The general form of solutions of higher-order heat equations has been given by D'Ovidio and Orsingher in \cite{Orsingher2012}, where it was shown that the solutions of the equation \eqref{1.6} corresponding to even and odd orders are substantially different in their structure and behavior.
We note also that in some cases the processes corresponding to the  solutions to \eqref{1.6} can be represented as compositions of Brownian motions (see \cite{HochbergOrsingher1996}).

\vspace*{0.2cm}
\noindent{\it 1.2. Higher-order heat-type equations with random initial conditions}\vspace*{0.2cm}

Origins of the study of PDE with random initial conditions by means of rigorous probabilistic tools can be traced back to the paper by Kamp\'{e} de Feriet \cite{KF}. Since then several approaches have been developed for investigation of solutions for various classes of PDE
subject to random initial conditions or random external forces.

Particular attention in the literature has been devoted to the study of  rescaled solutions of the heat, fractional heat, Burgers and some other equations with Gaussian and non-Gaussian initial conditions possessing weak or strong dependence. We refer, for example, to \cite{AL99}, \cite{LO}, \cite{Leonenko}, among many others (see also references in the book \cite{Leonenko}), spectral properties of random fields arising as approximations of the rescaled solutions are presented, e.g., in \cite{ALS03}, \cite{ALS06}.

Another approach to the investigation of solutions of PDE subject to random initial conditions has been developed in papers \cite{BKK}, \cite{BeiK}, \cite{KK}, \cite{KS04}, \cite{KV} and some others. Namely, the applicability of Fourier methods was studied, conditions were obtained under which the solutions  can be represented by uniformly convergent series, and also the methods were developed for the approximation of solutions by means of partial sums of the corresponding series, along with the conditions of convergence of the mentioned approximations in different functional spaces.

Odd-order heat-type equations of the form \eqref{1.6} subject to random initial conditions (represented by stationary processes) were studied in \cite{BKLO}, namely, the asymptotic behavior was analysed for the rescaled solution to the linear KdV equation with weakly dependent random initial conditions.

More general odd-order equations of the form
\begin{equation}\label{1.8}
\frac{\partial u}{\partial t}=\sum_{k=1}^{N}a_{k}\frac{\partial ^{2k+1}u}{%
\partial x^{2k+1}}\;,\qquad N=1,2,...,
\end{equation}
subject to the random initial condition represented by  harmonizable processes were considered in \cite{BKOS}.
Rigorous conditions were stated therein for the existence of the solutions in two cases, where  the initial condition is represented:  (1) by a strictly $\varphi$-subGaussian harmonizable process and (2) by a stochastic integral with respect to a process with independent increments.

We mention that equations of the form \eqref{1.8}, that is, evolution equations with odd order spatial derivatives, are also called dispersive equations, we refer to \cite{BKOS} for some discussion on their importance  for different applications (see also references therein). We note that these equations include, as the simplest case, the Airy equation having in the right hand side only one term with third order derivative.

Note that real-world applications more often require consideration of equations \eqref{1.1}, which are non-linear, however the study of linear equations of the form \eqref{1.8}  represents many points of interest by itself and also gives an important key for understanding the related equations with non-linear terms and for the corresponding numerical computations.

In the present paper we continue the investigation of solutions $U(t,x)$
 to the equation \eqref{1.8} subject to random initial condition
\begin{equation}
u(0,x)=\eta(x),
 \label{intr.3}
\end{equation}
where $\eta(x)$ is a real-valued harmonizable $\varphi$-sub-Gaussian process.

Note that the results obtained in the paper hold, in particular,  for the case of Gaussian initial condition.

The properties of sub-Gaussianity and, more generally, $\varphi$-sub-Gaussianity, are important features of stochastic processes, since they permit us to evaluate different functionals of these processes, and, in particular,  the behavior of their suprema. The theory developed for these processes provides us with powerful techniques and tools suitable not only for asymptotic results, but also for deriving many useful bounds for distributions of these processes. The general theory of $\varphi$-sub-Gaussian processes and fields is presented in \cite{BK,GKN,K1984,KOP,KO,KST,VKY2005,VKY2008}, simulation methods are given in \cite{KR,KRP, VKY2008}.

Our main aim in the present paper is to establish upper bounds for the distribution of the supremum of a solution $U(t,x)$ to the equation \eqref{1.8}, that is,  for $P\big\{\sup\limits_{(t,x)\in A} U(t,x)>u\big\}$, where $A$ is a bounded domain. We also consider the behavior of $\sup \frac{U(t,x)}{c(t)}$ over an unbounded domain, where $c(t)$ is some continuous function.

It is well known that the important practical aspect of the evaluation of statistical properties of physical systems relies on the efficient representation of the relation between solutions of corresponding PDE and the random initial condition.
Under the assumption that the initial condition is given by a $\varphi$-sub-Gaussian process, it is possible to state the rigorous conditions of the existence of a solution. And moreover, as we show in the present paper, it is possible to evaluate directly the tails of the distribution of supremum of the solution. The bounds obtained in the paper demonstrate clearly the relation between the $\varphi$-sub-Gaussian initial condition (random input) given in terms of a certain Orlicz function $\varphi$, and the behavior of the solution (random output), bounds for distribution of which are written in terms of $\varphi$ and its convex conjugate $\varphi^*$.

The paper is organized as follows. In Sections 2 and 3 we present all important definitions and facts on harmonizable $\varphi$-sub-Gaussian processes, which will be used for derivation of the main results. In Section 4 we reformulate and specify (as appropriate for the present paper) some results from \cite{BKOS} on the conditions of existence of solutions to \eqref{1.8} with the initial condition \eqref{intr.3}.  The main results are the bounds for the distributions of supremum of the solutions presented in Sections 5 and 6, accompanied by some examples.

\section{Harmonizable processes}

In the paper we will consider the random initial conditions given by harmonizable $\varphi$-sub-Gaussian processes. We present in this section and in the next one the necessary definitions and facts on these processes.

\begin{defi}\cite{L}
The second-order random function $X=\{X(t),t\in \R\},$ $ \Expect
X(t)=0$,  is called harmonizable if there exists a second-order random
 function $y={y(t), t\in \R}$, $\Expect y(t)=0$ such that the covariance
$\Gamma_y (t,s)=\Expect y(t)\overline{y(s)}$ has finite variation and $X(t)=\int_{\R}e^{itu}\,\d y(u)$,
where the  integral is defined in the mean-square sense.
\end{defi}

\begin{thm}\cite{L}[Loev theorem]\label{Loev_thm}
    The second-order random function $X=\{X(t), t\in \R\},$ $ \Expect X(t)=0$,
is harmonizable if and only if there exists a covariance
function $\Gamma_y (u,v)$ with finite variation such that
    $$\Gamma_x(t,s)= \Expect X(t) \overline {X(s)}=
\int_{\R}\int_{\R}e^{i(tu-sv)}\,\d \Gamma_y (u,v).$$
\end{thm}

\begin{rem}
Integral of the type $\int\int\limits_{A}f(t,s)\d g(t,s)$ is a common Lebesgue-Stieltjes integral, that is, a limit of the
sum $\sum\sum f(t,s)\Delta_i\Delta_j g(t,s)$, and integral of the type $\int\int\limits_{A}f(t,s)|\d g(t,s)|$ is a limit
of the sum $\sum\sum f(t,s)|\Delta_i\Delta_j g(t,s)|$ \cite{L}.
\end{rem}

\begin{defi}\label{real_harmonizable}
Real-valued second order random function  $X=\{X(t),t\in \R\}$ is called harmonizable,
if there exists a real-valued second order function  $y(u)$,
$ \Expect y(u)=0,$ $u\in \R$,  such that $X(t)=\int_{-\infty}^{\infty} \sin tu \,\d y(u)$
or $X(t)=\int_{-\infty}^{\infty} \cos tu \,\d y(u)$
and the covariance function $\Gamma_y(t,s)=\Expect y(t)y(s)$ has finite variation. The integral above is defined
in the mean-square sense.
\end{defi}

\begin{rem}
 In what follows, it is enough to consider one of the
  representations  $X(t)=\int_{-\infty}^{\infty} \sin tu \,\d y(u)$
or $X(t)=\int_{-\infty}^{\infty} \cos tu \,\d y(u)$, because all
proofs are similar for both cases, as well as for the case
$X(t)=\int_{-\infty}^{\infty}(a \sin tu + b\cos tu) \,\d y(u)$,
where $a$ and $b$ are some real constants.
\end{rem}

From Theorem~\ref{Loev_thm} the following theorem follows.
\begin{thm}\label{Loev_thm_real}
    Real-valued second order function   $X = \{X(t), t\in \R \}$, $ \Expect X(t)=0 $, is harmonizable if and only if
    there exists the covariance function  $\Gamma_y(u,v)$ with finite variation such that
    $$\Gamma_x(t,s)=\Expect X(t)X(s)=\int_{\R} \int_{\R} \cos(t u) \cos(s v) \,\d\Gamma_y (u,v)$$ or
$$\Gamma_x(t,s)=\Expect X(t)X(s)=\int_{\R} \int_{\R} \sin(t u) \sin(s v) \,\d\Gamma_y (u,v)$$
    \end{thm}

\begin{rem} If in Theorem~\ref{Loev_thm_real} or in Definition~\ref{real_harmonizable} the process
$y(u)$ is a process with uncorrelated increments such that
    $$\Expect (y(a)-y(b))^2 =F(a)-F(b)$$
    for $a>b$, where $F(x)$ is a monotonically increasing left-continuous function
    such that $$F(-\infty) = \lim_{\substack{
x\rightarrow-\infty \\
    }}F(x)=0, \quad F(+\infty) = \lim_{\substack{
        x\rightarrow+\infty \\
}}F(x)<\infty$$
(spectral function),
then $$\Gamma_x (t,s)=\int_{\R}\cos(t u) \cos(s u) \,\d F(u) \quad \mbox{or}
\quad \Gamma_x (t,s)=\int_{\R}\sin( t u) \sin(s u) \,\d F(u)$$
\end{rem}

\begin{rem}
    Real-valued stationary processes  $X(t), t\in\R$, $ \Expect X(t)=0$, with continuous covariance function
    $$\Gamma (t,s)= \int_{-\infty}^{\infty}\cos ((t-s)u)\,\d F(u),$$
      where $F(u)$ is a spectral function, can be considered as a sum of two harmonizable processes
      $$X(t)=\int_{-\infty}^{\infty}\cos tu \,\d\eta_1 (u)+\int_{-\infty}^{\infty}\sin tu \,\d\eta_2(u)$$
      where $\eta_1(u)$ and $\eta_2(u)$ are uncorrelated processes with uncorrelated increments
      such that
      $\Expect (\eta_i(a)-\eta_i(b))^2=F(a)-F(b)$ for $a>b$.
\end{rem}

\section{$\varphi$-sub-Gaussian random variables and processes}

We present now a short overview of basic facts from the theory of $\varphi$-sub-Gaussian random variables and processes.

\begin{defi} \cite{BK,Kr.Rut}
Let $\varphi=\{\varphi(x),x\in\R\}$ be a continuous even convex
function. The function $\varphi$ is an Orlicz N-function if
$\varphi(0)=0, \varphi(x)>0$ as $x\neq0$ and the following
conditions hold:
$\lim_{\substack{
        x\rightarrow 0 \\
}}\frac{\varphi(x)}{x}=0,\ \ \lim_{\substack{
x\rightarrow \infty \\
}}\frac{\varphi(x)}{x}=\infty.$
\end{defi}

\begin{defi} \cite{BK,Kr.Rut}
Let $\varphi=\{\varphi(x),x\in\R\}$ be an N-function. The function $\varphi^*$
defined by
 $$\varphi^*(x)=\sup_{\substack{
       y\in \R\\
}}(xy-\varphi(y))$$
 is called the Young-Fenchel transform (or convex conjugate) of the function $\varphi$.
 \end{defi}

 \begin{rem}\cite{BK, Kr.Rut}
    The Young-Fenchel transform of an N-function is again an N-function and the following inequality holds
    (Young-Fenchel inequality):
    $$xy\leq\varphi(x)+\varphi^*(y)\quad \text{as}\quad x>0,y>0.$$
 \end{rem}

\noindent\textbf{Condition Q.}\cite{GKN,KO}
  Let $\varphi$ be an N-function which satisfies
    $\lim \inf_{\substack{
            x\rightarrow 0 \\
    }}\frac{\varphi(x)}{x^2}=c>0,$ where
 the case $c=\infty$ is possible.

\begin{ex} Examples of the N-functions, for which condition Q holds:
\begin{align}
&\varphi(x)=\frac{|x|^{\alpha}}{\alpha}, \quad 1 <\alpha\leq2,\nonumber\\
&\varphi(x)=
{ \begin{cases}
 \frac{|x|^{\alpha}}{\alpha}, \quad |x|\geq 1,\alpha > 2;
\\
 \frac{|x|^{2}}{\alpha}, \quad |x|\leq1,\alpha > 2.\\
 \end{cases}}\label{ex_phi1}
 \end{align}
 We can also consider the following examples:
\begin{align}
 &\varphi (x)= \exp\{a|x|^{\alpha}\}-1, \alpha\leq2,\,\,a>0\nonumber\\
 &\varphi(x)=
 {\begin{cases}
\exp\{a|x|^{\alpha}\}- 1,\quad 1 <|x| = 1,\alpha >2, a>0,\\
 \exp\{a|x|^{2}\} - 1, \quad |x|<1,a>0.
\end{cases}}\label{ex_phi2}
\end{align}
\end{ex}

\begin{defi}\cite{GKN,KO}
  Let $\varphi$ be an $N$-function satisfying condition $Q$ and $\{\Omega, L,\P\}$ be a standard probability space.
  The random variable $\zeta$ belongs to the space $\Sub_\varphi(\Omega)$,
  if $\Expect \zeta=0,$ $ \Expect \exp\{\lambda\zeta\}$ exists for all $\lambda\in\R$ and there exists a constant
  $a>0$ such that the following inequality holds for all $\lambda\in\R$
$$\Expect \exp\{\lambda \zeta\} \leq \exp \{\varphi(\lambda a)\}.$$
\end{defi}

The space $ \Sub_\varphi(\Omega)$ is a Banach space with respect to
the norm \cite{GKN,KO}
$$\tau_\varphi (\zeta)= \sup_{\lambda \neq 0}\frac{\varphi^{(-1)}\left(\ln \Expect \exp(\lambda\zeta)\right)}{\left|\lambda\right|},$$
which can be written equivalently as
$$ \tau_\varphi (\zeta) = \inf\{a > 0: E \exp \{\lambda \zeta \} \leq
\exp\{ \varphi (a \lambda)\},$$
and  it is called the $\varphi$-sub-Gaussian standard of the random variable $\zeta$.

Examples of $\varphi$-sub-Gaussian random variables can be found
in the paper \cite{KR}, the books \cite{L} and
\cite{VKY2008}.

Centered Gaussian random variables  $\zeta= N(0,\sigma^2)$ are $\varphi$-sub-Gaussian with $\varphi (x)=\frac{x^2}{2}$
and $\tau^2_\varphi(\zeta)=\Expect \zeta^2=\sigma^2$. In the case, when  $\varphi(x)=\frac{x^2}{2}$,
$\varphi$-sub-Gaussian random variables are called sub-Gaussian.

The  important property of  $\varphi$-sub-Gaussian random variables is the exponential estimate for their tail probabilities, namely, if $\zeta$ is a $\varphi$-sub-Gaussian random variable, then for all $u>0$ we have
\begin{equation}\label{tail}
P\{|\zeta|>u\}\le 2\exp\left\{-\varphi^*\left(\frac{u}{\tau_\varphi(\zeta)}\right)\right\}.
\end{equation}

We will need some further properties of $\varphi$-sub-Gaussian variables.

\begin{defi}
\cite{KK} A family $\Delta$ of random variables ${\zeta} \in
\Sub_\varphi(\Omega)$ is called strictly $\varphi$-sub-Gaussian if
there exists a constant $C_\Delta $ such that for all countable sets $I$ of random
 variables ${\zeta_i}\in \Delta$, $i\in I$,  the following inequality holds:
\begin{equation}\label{*}
\tau_\varphi \left( \sum_{i \in I} \lambda_i \zeta_i \right) \leq
 C_\Delta \left( \Expect \left( \sum_{i \in I}\lambda_i \zeta_i \right)^2\right)^{1/2}.
\end{equation}
\end{defi}

The constant $C_\Delta$ is called the \textit{determining} constant of the family $\Delta$.

\begin{lmm}\cite{KK}
 The linear closure of a strictly $\varphi$-sub-Gaussian family
$\Delta$ in the space $L_2(\Omega)$ is the strictly
$\varphi$-sub-Gaussian with the same determining constant.
\end{lmm}
\begin{defi}
\cite{KK} Random process $\zeta=\{\zeta(t), t\in T\}$ is called
strictly $\varphi$-sub-Gaussian  if the family of random variables
$\{\zeta(t), t\in T\}$ is strictly $\varphi$-sub-Gaussian with a
determining constant~$C_\zeta$.
\end{defi}

\begin{ex} \cite{KK} Let a family of random variables $\{\xi_k, \; k = \overline{1, \infty}\}$
be a strictly $\varphi$-sub-Gaussian with determining constant
$C_\xi$, and let $X(t) = \sum\limits_{k=1}^\infty \xi_k \varphi_k (t)$,
where the series converges in mean square. Then the  random process
$X=\{X(t),t\in T\}$ is strictly $\varphi$-sub-Gaussian with determining
constant $C_\xi.$
\end{ex}

\begin{ex} \cite{KK} Let $\{\xi_k, \; k = \overline{1, \infty}\}$ be a family
 of independent random variables such that $\xi_k \in \Sub_\varphi (\Omega)$
 and $\varphi (x)$ be such function that $\kappa (x) = \varphi (\sqrt{x})$ is concave.
 If $\tau_\varphi (\xi_n ) \leq C ( E \xi_k^2 )^{1/2}$, $ C> 0$, and for
 all $t \in T$ the series
  $\sum\limits_{k=1}^\infty \xi_k^2 \varphi_k^2 (t)$ converges,
  then series $\sum\limits_{k=1}^\infty \xi_k \varphi_k (t)$,  $t \in T,$ is strictly
$\varphi$-sub-Gaussian random process with determining constant $C.$
\end{ex}
\begin{ex}\label{ex_kernel}\cite{KK} Let $K$ be a deterministic kernel and suppose that the process
$X=\left\{X(t), t\in T\right\}$
 can be represented in the form
$$
X(t) =\int\limits_T K(t,s) \,\d\xi(s),
$$
where $\xi(t)$, $t\in T$,  is a strictly $\varphi$-sub-Gaussian random process
and the integral above is defined in the mean-square sense. Then
the process $Xt)$, $t\in T$, is  strictly $\varphi$-sub-Gaussian random
process with the same determining constant.
\end{ex}

\begin{lmm}\label{lmm12}\cite{KR}, \cite{KRP} (p. 146)
 Let $Z(u),u\geq 0$ be a continuous, increasing function such
that $ Z(u)>0$ and the function $\frac{u}{Z(u)}$ is non-decreasing
for $u>u_0$, where $u_0\geq 0$ is a constant. Then for all
$u,v\neq 0$
\begin{equation}\label{**}
\left|\sin\frac{u}{v}\right|\leq \frac{Z\left(\left|u\right|+u_0\right)}{Z\left(\left|v\right|+u_0\right)}
\end{equation}
\end{lmm}

\begin{ex}\label{a1}
 If $Z(u)=u^\alpha, 0<\alpha\leq 1$, then
 $u_0=0, \left|\sin\frac{u}{v}\right|\leq \frac{\left|u\right|^\alpha}{\left|v\right|^\alpha}$.
\end{ex}

\begin{ex}\label{a2}
If $Z(u)=\ln^\alpha(u+1),\alpha>0,$ then $u_0=e^\alpha-1$ and
$\left|\sin\frac{u}{v}\right|\leq
\left(\frac{\ln\left(\left|u\right|+e^\alpha\right)}{\ln\left(\left|v\right|
+e^\alpha\right)}\right)^\alpha$.
\end{ex}

\begin{defi}
Function $Z(u), u\geq 0,$ is called \textit{admissible} for the space
$\Sub_\varphi(\Omega)$, if for $Z(u)$ conditions of Lemma~\ref{lmm12}
hold and for some $\varepsilon>0$ the integral
$$\int_0^\varepsilon\Psi\left(\ln\left(Z^{(-1)}\left(\frac{1}{s}
\right)-u_0\right)\right)ds$$ converges,
where $\Psi(v)=\frac{v}{\varphi^{(-1)}(v)}, v>0.$
\end{defi}

It is easy to see that functions $ Z(u) = u^\alpha$, \, $ 0 <
\alpha <1,$ and $Z(u) = \ln^\alpha (u+1)$, \, $\alpha >
\frac{1}{2},$ are admissible for the space $\Sub_\varphi(\Omega)$
if $\varphi (x)$ are defined in (\ref{ex_phi1}) and (\ref{ex_phi2})  respectively.

In order to derive our main results, we will use the estimates of distribution of suprema of $\varphi$-sub-Gaussian
random processes (see \cite{KS04}).

We consider a separable
$\varphi$-sub-Gaussian process defined on a separable metric space $(T,d)$, where  $T = \{ a_i \leq t \leq b_i, \, i=1, 2 \}$
and $ d({t}, {s}) = \max\limits_{i=1,2} |t_i -s_i|,$ ${t} =(t_1, t_2),$
${s} =(s_1, s_2).$

\begin{thm}\label{sup_X}
Assume that $ X=\{X(
{t}), {t} \in T \}$ is a separable
$\varphi$-sub-Gaussian process such that
\begin{equation}\label{suptau3.16}
\sup_{\substack{
        d({t},{s})\leq h,\\
        {t}, {s}\in T \\
}}\tau_{\varphi} (X ({t})-
X ({s}))\leq \sigma(h),
\end{equation}
where $\{\sigma(h),\; 0 < h \leq \max\limits_{i = 1,2} |b_i -a_i|\}$
is a monotonically increasing continuous function such
 that $\sigma(h)\rightarrow 0$ as $ h \rightarrow 0 $ and for some
 $\varepsilon > 0$
\begin{equation}\label{intPsi3.16}
 \int_0^\varepsilon \Psi \Big( \ln \frac{1}{\sigma^{(-1)} (u)} \Big)\,du < \infty,
 \end{equation}
where
$
\Psi (v) = \frac{v}{\varphi^{(-1)} (v)}. $
Then
$$ P \big\{ \sup_{ {t} \in T} |X( {t} )| > u \big\}
\leq   2A(u, \theta) $$
for all $ 0 < \theta <1 $  and
$$ u > \frac{2 I_\varphi (\min( \theta \varepsilon_0, \gamma_0))}{\theta (1-\theta)} ,$$
 where
$$ A(u, \theta) = \exp\Big\{ -\varphi^* \Big( \frac{1}{\ve_0} \Big( u(1-\theta) -
\frac{2}{\theta} I_\varphi (\min (\theta \varepsilon_0, \gamma_0)) \Big) \Big) \Big\},$$
and
$$\varepsilon_0 = \sup_{t \in T} \tau_\varphi (X({t} )), \quad
\gamma_0 = \sigma (\max_{i =1,2} |b_i -a_i| ) ,$$
$\varphi^* (u)$ is the Young-Fenchel transform of the function $\varphi$,
$$ I_\varphi (\delta) = \int_0^\delta \Psi \left( \ln \left[\Big(
\frac{b_1 -a_1}{2 \sigma^{(-1)} (u)} +1 \Big) \Big(
\frac{b_2 -a_2}{2 \sigma^{(-1)} (u)} +1 \Big)\right]\right)\,du.$$
\end{thm}
\begin{proof}
The statement of Theorem~\ref{sup_X} follows from the following general theorem proved in \cite{BK}:

{\it Theorem 4.2 (\cite{BK}, p.105).
Suppose that $X = \{ X(t), \, t \in T\}$ is a $\varphi$-sub-Gaussian process.
Let $\rho_X$ be the pseudometric  generated by $X$, that is,
$ \rho_X (t, s) = \tau_\varphi (X(t) - X(s)),$ $ t, s \in T,
$  and $\varepsilon_0 =\sup \tau_\varphi (X(t)) < \infty. $
Assume that $(T, \rho_X)$ is a separable space, the process $X$ is separable
on $(T, \rho_X)$ and
$$\int\limits_0^{\varepsilon_0}
\Psi (H(\varepsilon))\,d\varepsilon < \infty,$$ with $H(\varepsilon)$
being the metric entropy of the set $T$, that is, $H(\varepsilon) = \ln (N(\varepsilon)),$
where $N(\varepsilon)$ denotes the smallest number of  elements in an
$\varepsilon$-covering of the set $(T, \rho_X).$

Then
 $$ P \{\sup\limits_{t\in T} |X(t)| \geq u\} \leq 2A(u, \theta), $$
 for each $\theta \in (0, 1)$
and
$$
 u >\frac{2 \tilde{I}_\varphi ( \theta \varepsilon_0)}{\theta (1-\theta)},
$$
  where
 $$A(u, \theta) = \exp \Big\{- \varphi^* \Big( \frac{1}{\varepsilon_0}
\Big(u (1-\theta ) - \frac{2}{\theta} \tilde{I}_\varphi (\theta \varepsilon_0)
\Big) \Big) \Big\},\quad
 \tilde{I}_\varphi (\theta \varepsilon_0) = \int_0^{\theta \varepsilon_0}
  \Psi (H(\varepsilon))\,d\varepsilon.$$
}

Now one can see that the assertion of Theorem~\ref{sup_X} follows from the fact that
the process $X(t)$ is separable on $(T, \rho_X)$ and if $\sigma (u) < \gamma_0$
(that is, if $\varepsilon < \gamma_0)$ then
$$
H(\varepsilon) \leq \ln\Big[ \Big( \frac{b_1 -a_1}{2 \sigma^{(-1)} (\varepsilon)}
+ 1\Big) \Big( \frac{b_2 -a_2}{2 \sigma^{(-1)} (\varepsilon)}+ 1\Big) \Big]
\quad \mbox{and} \quad H(\varepsilon) = 0 \quad \mbox{if}
\quad \varepsilon \geq \gamma_0.
$$
\end{proof}

\section{Solutions of linear odd-order heat-type equations with random initial conditions}

The next theorem, which is a modification of Theorem~5.1 from the paper
\cite{BKOS},  gives the conditions of the existence of solution of odd-order heat-type equation with $\varphi$-sub-Gaussian initial condition.
\begin{thm}\label{thm_solution}
Let us consider the linear equation
\begin{equation}\label{linear_eqn}
\sum_{k=1}^N a_k\frac{\partial^{2k+1} U(t,x)}{\partial x^{2k+1}}=
\frac{\partial U(t,x)}{\partial t} ,\, t>0,\, x \in \R,\,\,
\end{equation}
subject to the random initial condition
\begin{equation}\label{random_ini_cond}
U(0,x)=\eta(x),\, x \in \R,
\end{equation}
and $\{a_k\}_{k=1}^N$ are some constants.

Let $\eta=\{\eta(x),x\in\R\}$ be a real harmonizable (as defined in definition~\ref{real_harmonizable})
and strictly $\varphi$-sub-Gaussian random process. Also
let $Z=\{Z(u),u\geq 0\}$ be a function admissible for the space
 $\Sub_\varphi(\Omega)$. Assume that the following integral converges
\begin{equation}\label{exists_classic}
\int_{\R}\int_{\R}\left|\lambda\right|^{2N+1}
\left|\mu\right|^{2N+1}Z\left(u_0+\left|\lambda\right|^{2N+1}
\right)Z\left(u_0+\left|\mu\right|^{2N+1}\right)d|\Gamma_y(\lambda,\mu)|<\infty.
\end{equation}
Then
\begin{equation}\label{solution}
 U(t,x)=\int_{-\infty}^{\infty}I(t,x,\lambda)\, dy(\lambda)
 \end{equation}
is the classical solution to the problem (\ref{linear_eqn})-(\ref{random_ini_cond}), where
\begin{equation}\label{Itxlm}
I(t,x,\lambda)=\kappa \left(\lambda x +t\sum_{k=1}^{N} a_k \lambda^{2k+1}(-1)^k \right),
\end{equation}
 and $\kappa (v) = \cos v$ or $\kappa (v) = \sin v$ for the cases when $\eta(x) = \int\limits_R cos (tu)\,dy(u)$  or $\eta(x) = \int\limits_R \sin(tu) \,dy(u)$ correspondingly.
 \end{thm}

\begin{rem}
Note that  under the condition \eqref{exists_classic}  all the integrals
\begin{equation}\label{integr}
\int_{R}\lambda ^{s}I\left( t,x,\lambda \right) dy\left( \lambda \right) ,\
s=0,1,2,\ldots ,2N+1,
\end{equation}
converge uniformly in  probability for $|x| \leq A$, $0 \leq t \leq T$ for all $A, T$ (we refer for more details to \cite{BKOS}).
\end{rem}

\begin{rem}
Let $\varphi \left( x\right) =\frac{\left| x\right|
^{p}}{p},$ $p>1$ for sufficiently large $x.$ Then the statement of Theorem
\ref{thm_solution} holds if the following integral converges
\begin{equation}\label{phixp}
\int_{R}\int_{R}\left| \lambda \mu \right| ^{2N+1}\left( \ln \left(
1+\lambda \right) \ln \left( 1+\mu \right) \right) ^{\alpha }d|\Gamma
_{y}\left( \lambda ,\mu \right)| ,
\end{equation}%
where $\alpha $ is a constant such that $\alpha >1-\frac{1}{p}$
(see  \cite{BKOS}).
\end{rem}


\medskip

Generalized solution for the equation \eqref{linear_eqn} with the random initial condition (\ref{random_ini_cond}) of the form
$\eta(x) = \int\limits_R cos (tu)\,dy(u)$  or  $\eta(x) = \int\limits_R \sin(tu) \,dy(u)$
is given by process
\begin{equation}\label{utx}
U(t,x) = \int_R I(t,x,\lambda)\, dy(u),
\end{equation}
where $I(t,x,\lambda)$ is given by \eqref{Itxlm}  (with $\kappa (x) = \cos x$ or
 $\kappa (x) = \sin x$), provided that the integral \eqref{utx} converges uniformly in
 probability for $|x| \leq A$, $0 \leq t \leq T$ for all $A, T$, and we do not require the uniform convergence of all the integrals \eqref{integr} as for the case of classical solution.

The next theorem presents the conditions of existence of the generalized solution and follows from the paper \cite{BKOS} (see Section 7 therein).

 \begin{thm}
 Let us consider the linear equation
 $$\sum_{k=1}^N a_k\frac{\partial^{2k+1} U(t,x)}{\partial x^{2k+1}}=
\frac{\partial U(t,x)}{\partial t} ,\, t>0,\, x \in \R,
 $$
subject to the random initial condition (\ref{random_ini_cond}).

 \noindent Let $Z$  be a function admissible for the space $\Sub_\varphi (\Omega),$
and let the following integral converge
\begin{equation}\label{integr-gen}
 \int_R \int_R Z \big( u_0 + |\lambda|^{2N+1}\big)
  Z \big( u_0 + |\mu|^{2N+1}\big)\, d(|\Gamma_y (\lambda, \mu|)).
  \end{equation}
  Then $U(t, x)$ defined in (\ref{solution}) is the generalized solution to the problem (\ref{linear_eqn})-(\ref{random_ini_cond}).
 \end{thm}

 \begin{rem}
 If the random initial condition $\eta(x)$ is a strictly $\varphi$-sub-Gaussian stationary process with the spectral function $F$, then condition \eqref{integr-gen} becomes
 \begin{equation*}
\int_{R}Z^{2}\left( u_{0}+\left\vert \lambda \right\vert
^{2N+1}\right) dF\left( \lambda \right) <\infty,
\end{equation*}
and conditions \eqref{exists_classic} and \eqref{phixp} can be modified in similar manner.
 \end{rem}

\section{On the distribution of supremum of solution \\ of the problem (\ref{linear_eqn})-(\ref{random_ini_cond})}

We now state the exponential bounds for the distribution of supremum of the field $U(t,x)$ representing the solution to (\ref{linear_eqn})-(\ref{random_ini_cond}).

\begin{thm}\label{thm_sup_solution} Let $y=\{y(u),\, u \in R\}$ be a strictly $\varphi$-sub-Gaussian random
process with a determining constant $C_y$ and $ U(t, x) =
\int\limits_{-\infty}^\infty I(t, x, \lambda)\,dy(\lambda)$, where
$I(t, x, \lambda)$ is given in Theorem~\ref{thm_solution}, $a \leq
t \leq b$, $c \leq x \leq d.$ Assume that $U(t, x)$ exists and is
continuous with probability one (this condition holds if
Theorem~\ref{thm_solution} holds). Let $\E y(t) y(s) = \Gamma_y
(s,t).$ Assume that $Z(u)$ is an admissible function for the space
$\Sub_\varphi (\Omega).$ If the integral
 \begin{align}
C_Z^2  = \int_{-\infty}^\infty \int_{-\infty}^\infty
&\left(Z \Big(\frac{|\lambda|}{2}+u_0\Big) + Z \Big( \frac{1}{2}
\Big| \sum_{k=1}^N a_k \lambda^{2k+1} (-1)^k\Big| + u_0 \Big)\right) \nonumber \\
\times &\left(Z \Big(\frac{|\mu|}{2}+u_0\Big) + Z \Big( \frac{1}{2}
\,\Big| \sum_{k=1}^N a_k \mu^{2k+1} (-1)^k\Big| + u_0 \Big)\right)\,d|\Gamma(\lambda, \mu)|
\end{align}
converges, then for $0<\theta <1$ and
$$ u > \frac{2 \hat{I}_\varphi (\min (\theta \Gamma, \gamma_0))}{\theta(1 - \theta)},$$
   the following inequality holds true
\begin{equation}\label{P_sup_U}
P \big\{ \sup_{\substack{a \le t \le b \\ c \le x \le d}}
|U(t, x)| > u \big\} \leq \exp \Big\{ - \varphi^* \Big(
\frac{1}{\Gamma} \Big( u(1-\theta) - \frac{2}{\theta}
\hat{I}_\varphi ( \min (\theta \Gamma, \gamma_0))\Big) \Big)
\Big\},
\end{equation}
where
\begin{equation}
\Gamma = C_y \int_{-\infty}^\infty \int_{-\infty}^\infty
\,d|F(u,v)|,
\end{equation}
\begin{align}\label{I_varphi}
\hat{I}_\varphi (\delta) & = \int_0^\delta
\Psi \left( \ln\left[ \Big( \frac{b-a}{2} \Big( Z^{(-1)}
\Big( \frac{2 C_Z C_y}{s} \Big) - u_0 \Big) +1 \Big) \right.\right.\nonumber \\
& \left.\left.\times \Big( \frac{c-d}{2} \Big( Z^{(-1)} \Big( \frac{2 C_Z
C_y}{s} \Big) -u_0 \Big) +1\Big)\right] \right) \,ds,
\end{align}
$$ \Psi (u) = \frac{u}{\varphi^{(-1)} (u)},  \quad
 \gamma_0 = \frac{2 C_y C_Z}{Z( \frac{1}{\varkappa} +u_0)}, \quad \varkappa = \max (b-a, d-c).$$
 \end{thm}

 \begin{proof} The assertion of this theorem follows from Theorem~\ref{sup_X}.

 Indeed,
 the process $ U(t, x)$  is separable since  $ U(t, s)$ is continuous with
 probability one. $U(t, x)$ is strictly $\varphi$-sub-Gaussian with
 the determining constant $C_y$, and, therefore, we can write:
 $$
  \sup_{ \substack {|t- t_1| \leq h \\ |x-x_1| \leq h \\}}
 \tau_\varphi ( U(t, x) - U(t_1 x_1))  \leq
C_y \left(E (U(t, x) - U (t_1, x_1))^2 \right)^{1/2}.
$$
We also have:
 \begin{align}
\varepsilon_0 & \leq C_y \sup \E |U(t, x|^2 \le  C_y
\int_{-\infty}^\infty \int_{-\infty}^\infty \big| I(t, x, \lambda
)
I(t, x, \mu) \big| \, d|\Gamma_y (\lambda, \mu)|  \nonumber \\
& \leq C_y \int_{-\infty}^\infty \int_{-\infty}^\infty \, d|\Gamma_y (\lambda, \mu)|
= C_y \, \Gamma.
\end{align}
Let us estimate now $\E (U(t,x) -U(t_1, x_1))^2$ for $\kappa (u) = \cos
(u).$
\begin{align}
\E (U(t, x) & - U(t_1, x_1))^2  = \left( \int_{-\infty}^\infty
(I(t, x, \lambda) -
I(t_1, x_1 , \lambda))\, dy(\lambda)\right)^2   \nonumber  \\
& = \int_{-\infty}^\infty \int_{-\infty}^\infty (I(t, x, \lambda) -
I(t_1, x_1 , \lambda))
 (I(t, x, \mu) - I(t_1, x_1 , \mu)) \, d \Gamma_y (\lambda, \mu) \nonumber \\
 & \leq \int_{-\infty}^\infty \int_{-\infty}^\infty |I(t, x, \lambda) -
I(t_1, x_1 , \lambda)|
 |I(t, x, \mu) - I(t_1, x_1 , \mu)| \, d |\Gamma_y (\lambda, \mu)|;
\end{align}
$|I(t, x, \lambda) - I(t_1, x_1 , \lambda)|
= |\cos A - \cos B|,$ where
$$ A = x\lambda + t \sum_{k=1}^N a_k \lambda^{2k+1} (-1)^k, \quad
B =  x_1\lambda + t_1 \sum_{k=1}^N a_k \lambda^{2k+1} (-1)^k. $$
Thus
$$ |I(t, x, \lambda) - I(t_1, x_1, \lambda) = 2 \Big| \sin\frac{A+ B}{2}
\sin \frac{B- A}{2} \Big| \leq 2 \Big| \sin \frac{B-A}{2}\Big|
 = 2 \Big|\sin (C+D)\Big|, $$
where
$$ C = \frac{\lambda (x_1 - x)}{2}, \quad D = \frac{t_1 - t}{2}
\sum_{k=1}^N a_k \lambda^{2k+1} (-1)^k. $$
Therefore,
\begin{align*}
2 | \sin(C+D)| & = 2| \sin C \cos D + \cos C \sin D|
\leq 2( |\sin C| + |\sin D|)  \\
& \leq 2 \Big( \Big| \sin \frac{\lambda (x_1 -x)}{2} \Big| + \Big|
\sin \frac{(t_1 - t)}{2} \sum_{k=1}^N a_k \lambda^{2k+1} (-1)^k
\Big| \Big).
\end{align*}
Let now $Z( x)$ be admissible function for the space $\Sub_\varphi
(\Omega).$  From Lemma~\ref{lmm12} it follows that
\begin{align*}
| I(t,x, \lambda) - I( t_1, x_1, \lambda|  &\leq
2 Z^{-1} \Big( \frac{1}{|x - x_1| }+ u_0\Big) Z \Big( \frac{|\lambda | }{2}+u_0 \Big)\\
&+ 2 \, Z^{-1} \Big(\frac{1}{|t- t_1|} +u_0 \Big)
Z \Big( \frac{1}{2} \Big| \sum_{k=1}^N a_k \lambda^{2k+1} (-1)^k
\Big| + u_0 \Big).
\end{align*}
Thus, we obtain:
\begin{align}
& \sup_{ \substack {|t- t_1| \leq h \\ |x-x_1| \leq h \\}}
 \tau_\varphi ( U(t, x) - U(t_1 x_1))  \leq
C_y \left(E (U(t, x) - U (t_1, x_1))^2 \right)^{1/2} \nonumber  \\
& \leq C_y \frac{2}{Z \big( \frac{1}{h} + u_0 \big)} \left[
\int_{-\infty}^\infty \int_{-\infty}^\infty \left( Z \Big(
\frac{|\lambda|}{2}  + u_0 \Big) + Z\Big( \frac{1}{2} \Big|
\sum_{k=1}^N a_k \lambda^{2k+1} (-1)^k \Big| + u_o\Big)
\right)\right.\nonumber  \\ & \times\left. \left( Z \big(
\frac{|\mu|}{2} +u_0\big) + Z\Big( \frac{1}{2} \Big| \sum_{k=1}^N
a_k \mu^{2k+1} (-1)^k \Big| + u_o \Big) \right)\, d\Gamma_y
(\lambda, \mu) \right]^{1/2} = C_y C_Z \frac{2}{Z \big(\frac{1}{h}
+u_0 \big)}.\label{(5.9)}
\end{align}
For $\kappa (u) = \sin u$ we have the same inequality. So, in the
notations of Theorem~\ref{sup_X}
$$ \sigma (h ) = 2 C_y C_Z \Big( Z\Big( \frac{1}{h} +u_0 \Big) \Big)^{-1},$$
that is,
$$ \sigma^{(-1)} (v) = \Big( Z^{(-1)} \Big( \frac{2 C_y C_Z}{v}\Big) - u_0 \Big)^{-1},
\quad 0 < v < \frac{2 C_y C_Z}{Z (\frac{1}{\varkappa} + u_0)} =
\gamma_0.$$
We now can conclude that conditions \eqref{suptau3.16} and \eqref{intPsi3.16} of Theorem \ref{sup_X} hold true.

\end{proof}

\begin{ex}\label{ex_Gaussian_Z1}
 Let $ y=\{y(u), u \in R\}$ be a centered Gaussian random process. Then $C_y = 1$,
$ \varphi (x) = \frac{x^2}{2}$, $ \varphi^{*} (x) = \frac{x^2}{2} $,
$\Psi(x) = \frac{1}{\sqrt{2}} x^{1/2}.$
Consider the following admissible function
$$ Z(u) = \ln^\alpha (u+1), \quad u \geq 0, \quad \alpha > 1/2.$$
In this case
$$ u_0 = e^{\alpha} - 1, \quad Z^{(-1)} (v) = \exp\left \{ v^{\frac{1}{\alpha}}\right\} - 1,
\quad Z(v + u_0) = \ln^\alpha (v + e^{\alpha}), $$
\vspace*{-0.3cm}
\begin{align}
C_Z^2 & = \int_{-\infty}^\infty \int_{-\infty}^\infty \Big( \ln^\alpha \Big(
\frac{|\lambda|}{2} + e^{\alpha} \Big) + \ln^\alpha \Big( \frac{1}{2}
\Big| \sum_{k=1}^N a_k \lambda^{2k+ 1} (-1)^k \Big| + e^{\alpha} \Big) \Big)
\nonumber   \\
& \times \Big( \ln^\alpha \Big(
\frac{|\mu|}{2} + e^{\alpha} \Big) + \ln^\alpha \Big( \frac{1}{2}
\Big| \sum_{k=1}^N a_k \mu^{2k+ 1} (-1)^k \Big| + e^{\alpha} \Big)
 \Big) \, d|\Gamma_y (\lambda, \mu)|.
\end{align}
 The above integral converges if the following integral converges
 \begin{equation}\label{int_ex1}
 \int_{-\infty}^\infty \int_{-\infty}^\infty \ln^\alpha(|\lambda | + e^{\alpha})
 \ln^\alpha(|\mu | + e^{\alpha}) \, d|\Gamma_y (\lambda, \mu)|.
 \end{equation}

That is, if condition (\ref{int_ex1}) holds true, then Theorem~\ref{thm_sup_solution}
holds.  It follows from (\ref{I_varphi}) that
\begin{align*}
\hat{I}_\varphi (\delta) & = \int_0^\delta \frac{1}{\sqrt{2}}
\left( \ln \left[\Big( \frac{b-a}{2} \Big( \exp \Big\{
 \Big(\frac{2 C_Z}{s} \Big)^{\frac{1}{\alpha}}\Big\} - e^{\alpha} \Big)
 + 1 \Big) \right.\right.   \\
& \left.\left.\times \Big( \frac{c-d}{2} \Big( \exp \Big\{ \Big( \frac{2 C_Z}{s}
  \Big)^{\frac{1}{\alpha}}\Big\} - e^{\alpha} \Big) +1 \Big)
  \right]\right)^{\frac{1}{2}} \, ds.
\end{align*}
Let now $ \frac{c-a}{2}\, e^{\alpha} > 1$ and $ \frac{b -a}{2} e^{\alpha} > 1$,
then
\begin{align}
\hat{I}_\varphi (\delta) & \leq \int_0^\delta \frac{1}{\sqrt{2}}
\Big( \ln \Big( \frac{c-d}{2} \frac{b-a}{2} \exp \Big\{ 2
\Big( \frac{2 C_Z}{s} \Big)^{\frac{1}{\alpha}} \Big\} \Big)
\Big)^{\frac{1}{2}} \, ds \nonumber  \\
& = \frac{1}{\sqrt{2}} \int_0^\delta
 \ln \Big( \frac{(c-d)(b-a)}{4} \Big)^{\frac{1}{2}} \, ds
+ \frac{1}{\sqrt{2}} \int_0^\delta
\Big( \frac{2 C_Z}{s}\Big)^{\frac{1}{2 \alpha}} \, ds
\nonumber   \\
& = \frac{\delta}{\sqrt{2}} \Big( \ln \Big( \frac{(c-d)(b-a)}{4} \Big)\Big)^{\frac{1}{2}}
+ \frac{\delta}{\sqrt{2}} \Big( \frac{2C_Z}{\delta} \Big)^{\frac{1}{2 \alpha}}
 {\Big(1 - \frac{1}{2 \alpha}\Big)^{-1}} .
\end{align}

It follows from (\ref{P_sup_U}) that in this case for
$$ u > \frac{2 \hat{I}_{\varphi} (\min (\theta \Gamma, \gamma_0 ) )}{\theta(1-\theta)}. $$
\begin{equation}
P \big\{ \sup_{\substack{a \leq t \leq b, \\
              c \leq x \leq d}}
| U(t, x)| > u \big\} \leq \exp \Big\{ -\frac{1}{2} \Big( \frac{1}{\Gamma}
\Big( u (1- \theta) - \frac{2}{\theta}
\hat{I}_\varphi ( \min (\theta \Gamma, \gamma_0))
\Big) \Big)^2 \Big\}.
\end{equation}
$$ \gamma_0 = \frac{2 C_Z}{\ln^\alpha \big( \frac{1}{\varkappa} + e^{\alpha} \big)} $$
If $\theta$ is such that $\theta \Gamma < \gamma_0 $, \;
 $ \big( \theta < \frac{\gamma_0}{\Gamma} \big) $ then for
 $$ u > \sup_{ 0 < \theta < \frac{\gamma_0}{\Gamma}}
 \frac{2 \hat{I}_\varphi (\theta \Gamma )}{\theta(1-\theta)} $$ we get the estimate
\begin{equation}
P \big\{ \sup_{\substack{a \leq t \leq b, \\
              c \leq x \leq d}}
| U(t, x)| > u \big\} \leq \inf_{0 < \theta < \frac{\gamma_0}{\Gamma}}
\exp \Big\{ -\frac{1}{2} \Big( \frac{1}{\Gamma}
\Big( u(1- \theta) - \frac{2}{\theta} {\hat{I}_\varphi } (\theta \Gamma)
\Big) \Big)^2 \Big\},
\end{equation}

\noindent and if  $ \frac{\gamma_0}{\Gamma} >1 $  \; then for
$$ u > \sup_{0< \theta <1}
\frac{2 \hat{I}_\varphi (\theta \Gamma )}{\theta(1-\theta)} $$
we get
\begin{equation}
P \big\{ \sup_{\substack{a \leq t \leq b, \\
              c \leq x \leq d}} |U (t, x)| > u \big\} \leq
\inf_{0 < \theta \leq 1} \exp \Big\{ -\frac{1}{2} \Big( \frac{1}{\Gamma}
\Big( u(1 - \theta) - \frac{2}{\theta} {\hat{I}_\varphi } (\theta \Gamma)
\Big) \Big)^2 \Big\}.
\end{equation}
\end{ex}

\begin{ex}
 Let $ y=\{y(u), u \in R\}$ be a centered Gaussian random process, as in example~\ref{ex_Gaussian_Z1}. Consider the
 admissible function $Z(u ) = | u|^\alpha $,
$0 < \alpha \leq 1.$ In this case
$$ u_0 = 0, \quad Z^{(-1)} (u) = u^{\frac{1}{\alpha}}, \quad u > 0 $$
\begin{align}
C_Z^2 & = \int_{-\infty}^\infty \int_{-\infty}^\infty \Big(
\Big(\frac{|\lambda |}{2} \Big)^\alpha + \Big| \frac{1}{2} \sum_{k=1}^N a_k
\lambda^{2k+1} (-1)^k \Big|^\alpha \Big)   \nonumber  \\
& \times \Big( \Big(\frac{|\mu |}{2} \Big)^\alpha + \Big| \frac{1}{2} \sum_{k=1}^N a_k
\mu^{2k+1} (-1)^k \Big|^\alpha \Big) \, d|\Gamma_y (\lambda , \mu)|  \nonumber \\
& \leq \int_{-\infty}^\infty \int_{-\infty}^\infty  \frac{1}{2^{2\alpha}}
\Big(|\lambda |^\alpha + \Big( \sum_{k=1}^N  |a_k|
|\lambda|^{2k+1} \Big)^\alpha \Big) \nonumber  \\
& \times \Big(|\mu |^\alpha + \Big( \sum_{k=1}^N  |a_k|
|\mu|^{2k+1} \Big)^\alpha \Big) \, d|\Gamma_y (\lambda, \mu)|.
\end{align}
This integral converges if the next integral converges
\begin{equation}\label{int_ex2}
\int_{-\infty}^\infty \int_{-\infty}^\infty  |\lambda \mu|^{(2N+1) \alpha}
\, d|\Gamma_y (\lambda, \mu)| < \infty
\end{equation}
That is, if condition (\ref{int_ex2}) holds, then theorem~\ref{thm_sup_solution} holds.
It follows from (\ref{I_varphi}) that
\begin{equation}
{\hat I}_\varphi (\delta) = \frac{1}{\sqrt{2}} \int_0^ \delta \Big( \ln \Big[
\Big( \frac{b-a}{2} \Big( \frac{2C_Z}{s} \Big)^{\frac{1}{\alpha}} +1 \Big)
\Big(\frac{c-d}{2} \Big( \frac{2C_Z}{s} \Big)^{\frac{1}{\alpha}} +1 \Big)
\Big]\Big)^{\frac{1}{2}} \, ds
\end{equation}
Since for $0 \leq \beta <1$, $ x>0,\, y>0$
\begin{align*}
\ln((1 +x)(1 + y)) & = \frac{1}{\beta} \ln [( 1 + x)( 1+ y)]^{\beta}
= \frac{1}{\beta} \big[\ln ( 1 + x)^\beta  + \ln ( 1+ y)^{\beta} \big]   \\
& \leq \frac{1}{\beta} \big[\ln ( 1 + x^\beta)  + \ln( 1+ y^{\beta}) \big]
\leq \frac{1}{\beta} \big( x^\beta  + y^{\beta} \big)
\end{align*}
in case of $\beta < \alpha$ we obtain
\begin{align}
\hat{I}_\varphi (\delta ) & \leq \frac{1}{\sqrt{2}} \int_0^\delta
\Big( \frac{1}{\beta} \Big( \frac{b- a}{2}
 \Big( \frac{2 C_Z}{s} \Big)^{\frac{1}{\alpha}} \Big)^\beta +
\frac{1}{\beta} \Big( \frac{c- d}{2}
 \Big( \frac{2 C_Z}{s} \Big)^{\frac{1}{\alpha}} \Big)^\beta
\Big)^{1/2} \, ds    \nonumber   \\
& \leq \frac{1}{\sqrt{2 \beta}} \Big(
\int_0^\delta
 \Big( \frac{b- a}{2}
 \Big( \frac{2 C_Z}{s} \Big)^{\frac{1}{\alpha}} \Big)^{\frac{\beta}{2}} \,ds
  + \int_0^\delta
 \Big( \frac{c- d}{2}
 \Big( \frac{2 C_Z}{s} \Big)^{\frac{1}{\alpha}}
 \Big)^{\frac{\beta}{2}} \, ds \Big)   \nonumber  \\
& =\frac{1}{\sqrt{2 \beta}}
  \int_0^\delta \Big( \Big (\frac{b- a}{2}\Big)^{\frac{\beta}{2}}
 \Big( \frac{2 C_Z}{s} \Big)^{\frac{\beta}{2\alpha}} +
 \Big( \frac{c- d}{2} \Big)^{\frac{\beta}{2}}
 \Big( \frac{2 C_Z}{s} \Big)^{\frac{\beta}{2\alpha}}
 \Big) \, ds     \nonumber  \\
 & = \frac{1}{\sqrt{2 \beta}}\big( 2 C_Z \big)^{\frac{\beta}{2 \alpha}}
  \, \delta^{(1- \frac{\beta}{2 \alpha})}
\frac{1}{1- \frac{\beta}{2 \alpha}}
\Big[ \Big( \frac{b- a}{2}\Big)^{\frac{\beta}{2}} +
\Big( \frac{c- d}{2}\Big)^{\frac{\beta}{2}}\Big] =
\hat{I}_\varphi (\delta, \beta ).
\end{align}
It follows from (\ref{P_sup_U})  that in this case for
$$ u > \frac{2 \hat{I}_{\varphi} (\min (\theta \Gamma, \gamma_0), \beta )}{\theta (1- \theta )},
\quad \gamma_0 = 2 C_Z \varkappa^\alpha $$ we get the estimate
\begin{equation}
P\big\{ \sup_{\substack {a\leq t \leq b \\
c \leq x <d \\}}  | U(t, x)| > u \big\}
  \leq \inf_{\theta,\beta} \exp \Big\{ -\frac{1}{2} \Big( \frac{1}{\Gamma}
 \Big(u(1 - \theta ) - \frac{2}{\theta}
 {\hat I}_\varphi (\min ( \theta\Gamma, \gamma_0), \beta )\Big) \Big)^2 \Big\}.
\end{equation}
\end{ex}

\begin{ex}
Let $y=\{y( u), \, u \in R\}$ be a $\varphi$-sub-Gaussian random process
 with
\begin{equation}
\varphi (x) =
\begin{cases}
\frac{ x^2}{\alpha},  \quad | x| \leq 1, \alpha > 2,
\\
\frac{| x|^\alpha}{\alpha},   \quad | x| \geq 1 , \alpha > 2. \\
\end{cases}
\nonumber
\end{equation}
In this case for $p$ such that $\frac{1}{p} + \frac{1}{\alpha} =1$
\begin{equation*}
{\varphi}^* (x) =
\begin{cases}
\alpha x^2/4,  \quad\quad \,\, 0 \leq |x| \leq {2 }/{\alpha},
\\
 |x| - 1/\alpha,   \quad  2/\alpha< |x| \leq 1, \\
 {x^p}/{p},     \quad \quad\quad |x| > 1, \\
\end{cases}
\nonumber \quad \text{and} \quad
\Psi (u) =
\begin{cases}
\frac{1}{\alpha^{1/2}} u^{1/2},  \quad\,\, 0< u < \frac{1}{\alpha},
\\
\frac{1}{\alpha^{1/\alpha}} u^{1 - \frac{1}{\alpha}},   \quad u > \frac{1}{\alpha}. \\
\end{cases}
\nonumber
\end{equation*}
Let $Z ( u)$ be admissible function for this space. Then for
$$ u > \max \Big( 1,\frac{2 {\hat I}_\varphi (\min (\theta \Gamma, \gamma_0))}{\theta(1-\theta)}
\Big),
  $$
\begin{equation*}
P \Big\{ \sup_{\substack{a \leq t \leq b \\
      c \leq x \leq d}}
 |U ( t, x) | > u \Big\} \leq 2
\exp\Big\{ -\frac{1}{p} \Big( \frac{1}{\Gamma} \Big( u ( 1-\theta) -
\frac{2}{\theta} {\hat I}_\varphi (\min (\theta \Gamma, \gamma_0)) \Big) \Big)^p \Big\}.
\end{equation*}
\end{ex}

\section{Rate of growth of the field \ $U(t, x)$\ \ on an unbounded  domain}

To evaluate the rate of growth of the  the field $U(t, x)$ on an unbounded domain, we will use the next theorem, which is a direct corollary of Theorem~3 from
the paper \cite{KST}.
\begin{thm}\label{thm_increase_rate_xi} Let $\{ \xi(x, t), (x, t) \in V\}$,
$ V =[-A, A]\times [0, + \infty]$, be a separable strictly
$\varphi$-sub-Gaussian random field with a determining constant $C_\xi $.
Assume that the following conditions are satisfied:
\begin{enumerate}
\item[1.]
$\{[b_k, b_{k+1}], k= 0, 1, \ldots\}$  is a family of such  segments
that $b_0=0$, $ 0 < b_k < b_{k+1} < +\infty $, $ k \ge1$,
$V_k = [-A, A] \times [b_k, b_{k+1}]$, $ \bigcup_k { V_k} = V;$

\item[2.]
 There exist the increasing continuous functions $\sigma_k (h)$, $h>0$,
 such that $\sigma_k (h) \to 0$ as $h \to 0$,
\begin{equation}
\sup_{\substack{|x - x_1|\leq h \\ |t - t_1| \leq h \\
(x, t), (x_1,t_1) \in V_k }}
\big( E (\xi (x, t) - \xi (x_1, t_1) )^2 \big)^{1/2} \le \sigma_k (h),
\end{equation}
 and for $ k = 0, 1, \ldots,$ and some  $\delta >0$
\begin{equation}\label{int_Psi}
\int_0^\delta  \Psi \left( \ln \left( \frac{1}{\sigma_k^{(-1)} (u)}
\right) \right) \, du < \infty;
\end{equation}

\item[3.]
$ c(t)$, $t \in R$,  is a continuous function such that $c( t) >0$,
$ t \in R$, and  let $c_k := \min \limits_{t \in [b_k, b_{k+1}]} c (t )$;

\item[{4}.]
$ \varepsilon_k =
C_\xi \sup\limits_{x, t \in V_k} \Big( E |\xi (x,t)|^2 \Big)^{1/2},
\quad k=0,1,2,\cdots \quad {\mbox and} \quad
\sup\limits_{k =\overline{0,\infty}} \frac{\varepsilon_k}{c_k} < \infty$

\item[{5}.]
For $\delta >0$
$$ I_{\varphi, k} (\delta) = C_\xi \int_0^{\frac{\delta}{C_\xi}} \Psi
\left( \ln \left( \frac{A}{{\sigma_k^{(-1)}} (u)} +1 \right) +
\ln \left( \frac{b_{k+1} - b_k}{{2 \sigma_k^{(-1)}} (u)} +1 \right)
\right) \, du $$
 and for some $\theta$, $ 0 < \theta < 1 ,$  $\sup\limits_{k = \overline{0, \infty}}
\frac{I_{\varphi, k} (\theta \varepsilon)}{c_k} < \infty$.

\item[{6}.]
 The series
$ \sum\limits_{k=0}^\infty \exp \Big\{ -\varphi^{*} \Big(
 \frac{s c_k (1-\theta)}{2 \varepsilon_k} \Big) \Big\}$
 converges for some $s$  such that
 $ \sup\limits_{k = \overline{0, \infty}}
 \frac{4 \varepsilon_k}{ c_k (1 - \theta)} < s < \frac{u}{2},$ where $u$ satisfies \eqref{uu}.
 \end{enumerate}
Then the increase rate of the field $\xi(x, t)$ can be estimated as follows:
$$ P \left\{ \sup \frac{|\xi(x, t)|}{ c(t)} > u \right\} \leq
2 \exp \left\{ - \varphi^{*} \left(\frac{u}{s} \right) \right\} \times
\sum_{k=0}^\infty \exp \Big \{ -\varphi^{*} \Big(
 \frac{s c_k (1- \theta)}{2 \varepsilon_k} \Big) \Big\}  =: 2 A(u ) $$
  for
\begin{equation}\label{uu} u > \sup_{k = \overline{0,\infty}}
\frac{I_{\varphi, k} (\theta \varepsilon_k)}{c_k} \frac{4}{\theta (1-\theta)}\,.\end{equation}
\end{thm}

Applying Theorem \ref{thm_increase_rate_xi}, we are able to evaluate the behavior of the field $U(t,x)$.

\begin{thm}\label{thm_increase_rate_U}
Let $y=\{y(u), u \in R\} $ be a strictly $\varphi$-sub-Gaussian random process
with a determining constant $C_y$ and
$$ U(t, x) = \int_{-\infty}^\infty I (t, x, \lambda ) \, dy(u), \quad
\textit{where} \quad (x,t) \in V, \; V= [-A, A] \times [0, +\infty], \;
A > 0,$$
$I (t,x, \lambda )$ is given in theorem~\ref{thm_solution} in (\ref{Itxlm}). Assume that $U(t, x)$
exists and is continuous with probability one ( this condition holds
if the conditions of the theorems~\ref{thm_solution} and \ref{thm_sup_solution} hold).

Let $ E y( t) y(s) = \Gamma_y (t, s),$  $ Z(u )$ be an admissible
function for the space $\Sub_\varphi (\Omega )$,  the integral
\begin{eqnarray}
C_Z^2 & =& \int_{-\infty}^\infty \int_{-\infty}^\infty \left(  Z \Bigl(
\frac{|\lambda |}{2} +u_0 \Bigr) + Z \Big( \frac{1}{2} \Big|
\sum_{k=1}^N a_k \lambda^{2 k + 1} (-1)^k \Big| + u_0 \Big) \right)
  \nonumber  \\
  && \times   \left(Z  \Big(
\frac{|\mu |}{2} +u_0 \Big) + Z \Big( \frac{1}{2} \Big|
\sum_{k=1}^N a_k \mu^{2 k + 1} (-1)^k \Big| + u_0 \Big)
 \right)
 \, du|\Gamma (\lambda, \mu)|
\end{eqnarray}
converge and the following conditions hold:
\begin{enumerate}
\item[{1}.]  $\{[b_k, b_{k+1}], \, k= 0, 1, \ldots \} $ is a family
 of such  segments
that $b_0=0$, $ 0 < b_k < b_{k+1} < +\infty $, $ k = 1, 2, \ldots$,
$b_{k+1} - b_k \geq 2 A,$
$V_k = [-A, A] \times [b_k, b_{k+1}]$;

\item[{2}.] $ c(t),  t>0$ is a continuous function such
that
$ c( t ) > 0$, $t \leq 0$, and let
$c_k = \min_{t \in [b_k, b_{k+1}]} c(t)$;

\item[{3}.] $$ \varepsilon_k = C_y \sup_{x, t \in V_k}
\big( E ( | U(t, x)|)^2 \big)^{1/2}, \quad k = 0, 1, 2,\ldots, \quad
\textit{and} \quad \sup_{ k= \overline{0, \infty}} \frac{\varepsilon_k}{c_k}
 < \infty;$$

\item[{4}.] Let for $\delta >0$
\begin{align}
\hat{I}_{\varphi, k} (\delta)  = C_{\xi} \int_0^\delta \Psi
 \Bigg( &\ln \left( A \Big( Z^{(-1)} \Big( \frac{2 C_Z C_y}{s} - u_0 \Big)
+ 1\Big) \right)    \nonumber \\
 +&\ln \left( \frac{b_{k+1} -b_k}{2} \Big( Z^{(-1)} \Big(
\frac{2 C_Z C_y}{s}\Big) - u_0 \Big) +1 \right) \Bigg) \, ds
\end{align}
and for some $\theta, \; 0 < \theta < 1$
$$ \sup_{k=\overline{0, \infty}}
\frac{\hat{I}_{\varphi, k} (\theta \varepsilon_k)}{c_k} < \infty.$$

\item[{5}.] The series
$ \sum\limits_{k=0}^\infty \exp \Big\{-\varphi^{*} \Big(
\frac{s c_k (1-\theta)}{2 \varepsilon_k} \Big) \Big\} $
converges for some $s$ such  that
$ \sup\limits_{k=\overline{0, \infty}} \frac{4 \varepsilon_k}{c_k (1- \theta)}
 < s < \frac{u}{2}. $
 \end{enumerate}
 Then for
\begin{equation}
  u> \sup_{k=\overline{0, \infty}}
 \frac{I_\varphi (\theta \varepsilon_k)}{c_k}
 \frac{4}{\theta (1 - \theta)}
\end{equation}
we get that the increase rate of $U (t, x)$ can be estimated as follows
\begin{equation}
P\Big\{ \sup_{(x, t) \in V}  \frac{|U (t, x)|}{c(t)} > u \Big\}
\leq 2 \exp\Big\{ - \varphi^{*} \Big( \frac{u}{s} \Big) \Big\}
\sum_{k=0}^\infty \exp\Big\{ -\varphi^{*} \Big(
\frac{s c_k (1- \theta)}{2 \varepsilon_k} \Big) \Big\} =
2 A_U ( u ).
\end{equation}
\end{thm}

\medskip

\begin{cor}\label{cor_53} Let the assumptions of Theorem~\ref{thm_increase_rate_U} hold true.
Then there exists a random variable $\xi$  such that $$P \{ \xi > u\} \leq 2 A_U (u)$$ for
 $$
 u > \sup_{k=\overline{0, \infty}}
 \frac{I_\varphi (\theta, \varepsilon_k)}{c_k}
\frac{u}{\theta (1 - \theta)}  $$
 and
$$|U(x, t)| < \xi c(t)$$ with probability one.
\end{cor}

\noindent {\it Proof of Theorem~\ref{thm_increase_rate_U}.} This theorem follows
from Theorem~\ref{thm_increase_rate_xi}. Indeed the process $U(t, x) $ is continuous,
that is, this process is separable. From \eqref{(5.9)} it follows that
$$  \sigma_k (h) = 2 C_y C_Z \Big( Z \Big( \frac{1}{h_k} +u_0 \Big)
\Big)^{-1} \quad \textrm {for} \quad 0< h_k < b_{k+1} - b_k $$
and
$$ \sigma_k^{(-1)} (t) =
 \frac{1}{Z^{(-1)} \big( \frac{2C_y C_Z}{t}\big) -u_0}\,, \quad
 t< \gamma_{k 0}, \quad
 \gamma_{k 0} = 2 C_y C_Z Z \Big( \frac{1}{b_{k+1} - b_k} + u_0 \Big).
 $$
Assumption (\ref{int_Psi}) holds, since
$$
\int_0^\delta \Psi \Big( \ln \Big( \frac{1}{\sigma_k^{(-1)}(u)}
\Big) \Big) \, du = \int_0^\delta \Psi \Big( \ln \Big(
Z^{(-1)} \Big( \frac{2C_y C_Z}{u} \Big) -u_0 \Big) \Big) \, du $$
and the function $Z( x )$ is admissible for the space
$\Sub_\varphi (\Omega)$.
\begin{flushright} $\square$
\end{flushright}

\begin{ex} Let $y=\{y( u ),u\in\R\}$ be a centered Gaussian random process,
$Z( u) = \ln^\alpha (u+1)$, $c(t) >0$ be an increasing  function and
$c_k =  c(b_k)$. In this case the process $U (t, s)$ is centered Gaussian
process as well, $C_y = 1$  and
\begin{align*}
\varepsilon_k & = \sup_{(x, t) \in V_k} \big( E ( U(t, x))^2 \big)^{1/2}   \\
& = \sup_{(x, t) \in V_k} \Big( \int_{-\infty}^\infty \int_{-\infty}^\infty
I(t, x, \lambda ) I (t, x, \mu ) \, d\Gamma (\lambda, \mu) \Big)^{\frac{1}{2}}
\\ & \leq \Big( \int_{-\infty}^\infty
\int_{-\infty}^\infty \, d|\Gamma(\lambda, \mu )|\Big)^{\frac{1}{2}} =
\Gamma^{\frac{1}{2}}
\end{align*}
Therefore, the assumption 3 of Theorem \ref{thm_increase_rate_U} holds true.

The series in the assumption 5 has now the following form:
\begin{equation}\label{57}
 \sum_{k=0}^\infty \exp \Big\{ - \frac{1}{2} \Big(
 \frac{s c_k (1- \theta)}{2\varepsilon_k} \Big)^2 \Big\} \leq
\sum_{k=0}^\infty \exp \Big\{ - \frac{1}{2} \Big(
 \frac{s c_k (1- \theta)}{2 \Gamma^{1/2}} \Big)^2 \Big\}
\end{equation}
and if the above series converges for $u > 2 s$ and for
\begin{equation}
u > \sup_{k = \overline{0,\infty}}
\frac{I_\varphi (\theta \Gamma^{\frac{1}{2}})}{c_k}
\frac{4}{\theta (1- \theta)} =
\frac{I_\varphi (\theta \Gamma^{\frac{1}{2}})}{c_0}
\frac{4}{\theta (1- \theta)}
\end{equation}
we have
\begin{equation}\label{59}
P \Big\{ \sup_{(x, t) \in V}
\frac{U(t, x)}{c(t)} > u \Big\}
\leq 2 \exp \Big\{ -  \frac{u^2}{2s^2} \Big\}
\sum_{k=0}^\infty \exp \Big\{ -\frac{1}{2}
\Big( \frac{s c_k (1- \theta)}{2 \Gamma^{\frac{1}{2}}} \Big)^2 \Big\}
= 2 \hat{A} (u).
\end{equation}
Let now $b_k =L e^k$ and $2 A < e (e- 1) L,$  that is,
$2 A < b_{k+1} - b_k$, $k= \overline{0, \infty}.$
Let
$$ c(t) = D_\delta \Big( \ln \Big( \ln \frac{t}{L} \Big) \Big)^{\frac{1}{2}}\, ,
\quad L > \frac{2 A}{e (e-1)} \, ,\quad
D_\delta = \frac{2 \Gamma^{\frac{1}{2}} (2 (1 + \delta))^{\frac{1}{2}}}{s (1 - \theta)}
$$
$\delta >0$ is any number, then series in (\ref{57}) converges, that is,
the inequality (\ref{59}) holds true.

From Corollary~\ref{cor_53} it follows that
 in this case with probability one $| U(t, x)| < c(t) \xi $, where
 $\xi $ is random variable such that $P\{ \xi > u \} \leq 2 \hat{A} (u)$ for
 $ u > {I_\varphi \big(\theta \Gamma^{\frac{1}{2}}\big)}
 \frac{4}{L\theta (1 -\theta)}.$
\end{ex}

\end{document}